\documentclass[a4paper, 11pt]{article}
\usepackage{amsmath,amssymb,esint,amscd,xspace,fancyhdr,color,authblk,srcltx,fontenc,bbm,amsxtra,latexsym}
\usepackage[mathscr]{eucal}
\usepackage{indentfirst}
\usepackage{hyperref}

\newtheorem{theorem}{Theorem}[section]

\newtheorem{definition}[theorem]{Definition}
\newtheorem{lemma}[theorem]{Lemma}
\newtheorem{proposition}[theorem]{Proposition}
\newtheorem{remark}[theorem]{Remark}
\newenvironment{proof}[1][Proof]{\textbf{#1.} }{\hfill\rule{0.5em}{0.5em}}
{\catcode`\@=11\global\let\AddToReset=\@addtoreset
\AddToReset{equation}{section}

\AddToReset{theorem}{section}

\newcommand{\R}{\mathbb R} 
\newcommand{\N}{\mathbb N}

\newcommand{\tp}{\displaystyle \int \limits}

\newcommand{\su}{\mathop{\sup} \limits}

\newcommand{\vc}{\infty}
\newcommand{\na}{\nabla}
\newcommand{\q}{\omega}

\newcommand{\Lo}{L^1_{\operatorname{loc}}}

\newcommand{\p}{\mathscr{R}^{\La}_\beta}

\newcommand{\Del}{\Delta}
\newcommand{\V}{\mathbf V}
\newcommand{\La}{\mathcal L}
\newcommand{\Ri}{\mathscr R_\La}

\newcommand{\K}{\mathbf K}
\newcommand{\T}{\Ri}

\newcommand{\ls}{\lesssim}

\newcommand{\rh}{\rho^{-1}(y)}

\newcommand{\Mor}{\mathbb{M}^{p,s}_{\alpha,\theta}(\q)}
\newcommand{\Mo}{\mathbb{M}^{p,s}_{\alpha,\theta}(\q,\nu)}

\title{On boundedness property of singular integral operators associated to a Schr\"odinger operator in a  generalized Morrey space and applications}
\author{Le Xuan Truong$^{1}$, Nguyen Thanh Nhan$^2$, Nguyen Ngoc Trong$^{3,\star}$ \\ {\small $^1$Division of Computational Mathematics and Engineering, Institute for Computational Science, Ton Duc Thang
		University, Ho Chi Minh City, Vietnam\\
		Faculty of Mathematics and Statistics, Ton Duc Thang University, Ho Chi Minh City, Vietnam \\   Email: \texttt{lexuantruong@tdtu.edu.vn }}\medskip\\ {\small $^2$Department of Mathematics,  Ho Chi Minh City University of Education,\\ Ho Chi Minh City, Viet Nam \\ Email: \texttt{nhannt@hcmue.edu.vn}} \medskip\\ {\small $^3$Department of Primary Education, Ho Chi Minh City University of Education, Ho Chi Minh City, Vietnam \\  $^{\star}$ Corresponding author\\ Email: \texttt{trongnn@hcmue.edu.vn}}} 
\date{\today}

\begin{document}
 
\maketitle
\begin{abstract}
In this paper, we provide the boundedness property of the Riesz transforms associated to the Schr\"odinger operator $\La=-\Delta + \V$ in a new weighted Morrey space which is the generalized version of many previous Morrey type spaces. The additional potential $\V$ considered in this paper is a non-negative function satisfying the suitable reverse H\"older's inequality. Our results are new and general in many cases of problems. As an application of the boundedness property of these singular integral operators, we obtain some regularity results of solutions to Schr\"odinger equations in the new Morrey space.

\medskip

\noindent 

\medskip


\end{abstract}   
                  
\section{Introduction}\label{sec:intro}

In 1938, the classical Morrey space was firstly introduced by Charles B. Morrey in~\cite{m1} for studying the second order elliptic equations. Several standard properties of Morrey space can be found in~\cite{aa1,duong1} and~\cite{x1}. The advantage of using this functional space lies in the fact that ones can obtain better regularity properties for solutions of the boundary elliptic and parabolic equations in Morrey space. However, the regularity results for many partial differential equations can be provided as applications of the boundedness properties of several singular integral operators. By these interesting applications, many mathematicians considered the boundedness properties of singular integral operators in different kinds of functional spaces so called Morrey type spaces.



Recently, many authors have considered this kind of problem by extending to several weighted Morrey spaces, for instance~\cite{duong1},~\cite{k1} and~\cite{Feuto2014}. They have showed that the singular integral operators are not only bounded in weighted Lebesgue spaces but also in weighted Morrey spaces. In addition, lots of Morrey type spaces associated to a Schr\"{o}dinger operator have been also studied (see~\cite{a,li,Y.Liu,ren,t1}) to extend the well-known Morrey spaces. In recent years the problem related to Schr\"{o}dinger operator has attracted a great deal of  attention of many mathematicians; see~\cite{a2,a1, b1, b2,c1,d1,d2,l1,Shen1995,z1} and references therein. 

Motivated by these works, we consider in this paper the boundedness property of some singular integral operators associated to a Schr\"odinger operator  $\La=-\Delta +\V$ on $\mathbb{R}^{n}$, $n\geq 3$ in new generalized Morrey spaces, where the potential $\V$ belongs to $RH_q$ for some $q>n/2$, i.e., there exists a constant $C = C(q, \V) > 0$ such that the reverse H\"{o}lder's inequality
$$
\Big(\frac{1}{|B|}\int\limits_{B}\V(x)^q dx\Big)^{\frac{1}{q}} \leq \frac{C}{|B|}\int\limits_{B}\V (x)dx
$$
holds for every ball $B \subset \mathbb{R}^n$.  More precisely, we establish the boundedness property of the $\La$-Riesz transform  $\mathscr{R}_{\La}=D^2 \La^{-1}$ and the  $\La$-fractional Riesz transform  $\mathscr{R}^{\La}_\beta=D \La^{-\beta}$ in new Morrey type spaces $\Mor$ and $\Mo$, respectively. The regularity result of solutions to Schr\"odinger type equations in these functional spaces is also obtained as an application. We note that all notations and definitions will be introduced in the next section.

The boundedness property of the $\La$-Riesz transform $\mathscr{R}_{\La}$ in $L^p$ has been studied by Zhong \cite{z1} with a non-negative polynomial $\V$ and by Shen \cite{Shen1995} if $\V\in RH_{n/2}$. On the other hand, the boundedness of the $\La$-fractional Riesz transform $\mathscr{R}^{\La}_\beta$ from  $L^{p}(\q^p)$ into $L^q(\q^q)$ has been proposed by Sugano~\cite{Sugano}. With our knowledge, the boundedness property of two above operators $\mathscr{R}_{\La}$ and $\mathscr{R}^{\La}_\beta$ have never been studied in our Morrey type space $\Mo$ even in the classical Morrey space $\mathbf{M}_{\theta/p}^{p}$. Hence, we believe that the results in this paper are general in many cases of the problem.

Moreover, we emphasize here that the space $\Mo$ in our paper is a generalized version of many well-known Morrey type spaces. 
\begin{itemize}
\item In the case when $\mathbb{V}= 0$, $\q = \nu = 1$, $s=\infty$ and $\alpha =0$, the space $\Mo$ becomes to the classical Morrey  $\mathbf{M}_{\theta/p}^{p}$. 
\item In 1988, Fofana \cite{f3} proposed an extension of the classical Morrey space (see~\cite{f1,f2}) the space $(L^{p},L^{s})^{\theta }$ as follows
\begin{align*}
\left\Vert f\right\Vert _{(L^{p},L^{s})^{\theta }} =\underset{r>0}{\sup }\left[ \int\limits_{\mathbb{R}^{n}}\Big( |B(x,r)|^{-\theta }\left\Vert f\chi _{B(x,r)}\right\Vert _{L^{p}(\R^n)} \Big) ^{s}dx\right] ^{1/s}<\infty.
\end{align*}
When $\mathbb{V}= 0$, $\q = \nu = 1$ and $\alpha =0$, the space $\Mo$ coincides to $(L^{p},L^{s})^{\theta }$.
\item In 2009, Komori and Shirai \cite{KS2009} introduced the Morrey type space $M_{q}^{p}\left(\q,\nu\right)$ with two Muckenhoupt weights $\q$ and $\nu$ as
	\begin{align*}
\left\| f \right\|_{M_q^p\left( {\q,\nu} \right)}^p  = \mathop {\sup }\limits_{B\left( {{x_0},r} \right) \subset \R^n } \nu{\left( {B\left( {{x_0},r} \right)} \right)^{1/p - 1/q}}\int\limits_{B\left( {{x_0},r} \right)} {{{\left| {f\left( x \right)} \right|}^p}\q\left( x \right)dx < \infty. }
	\end{align*}
When  $\mathbb{V}=0$, $s=\infty$, $\theta =1/q-1/p$ and $\alpha =0$, two Morrey spaces $M_{q}^{p}\left( \q,\nu\right)$ and $\Mo$ are exactly the same.
\item In 2009, Tang and Dong \cite{t1} defined a Morrey space $L^{p,\lambda}_{\alpha, \V}$ associated to Schr\"odinger operator by
	\begin{equation*}
	\left\| f \right\|_{L_{\alpha ,\V}^{p,\lambda }}^p = \mathop {\sup }\limits_{B\left( {{x_0},r} \right) \subset \R^n } {\left( {1 + \frac{r}{{\rho \left( {{x_0}} \right)}}} \right)^\alpha }{r^{ - \lambda }}\int\limits_{B\left( {{x_0},r} \right)} {{{\left| {f\left( x \right)} \right|}^p}dx < \infty }.
	\end{equation*}
When $s=\infty$ and $\q = \nu = 1$, the space $\Mo$ is also $L_{\alpha p,\V}^{p,n\theta p}$.	
\item Later in 2014, Feuto \cite{Feuto2014} extended a Morrey type space $(L^p(\q),L^s)^\alpha$ equipped to the norm
	\begin{align*}
	 \left\Vert f\right\Vert _{(L^p(\q),L^s)^\alpha} = \su_{r>0} \Big[ \int\limits_{\mathbb{R}^{n}}\Big( \q \left(	B(x,r)\right)^{\frac{1}{\alpha}-\frac{1}{p}-\frac{1}{s} }\left\Vert f\chi _{B(x,r)}\right\Vert _{L^{p}\left( \q \right)} \Big) ^{s}dx\Big] ^{1/s}<\infty .
	\end{align*}	
It is easy to see that if $\V= 0$, $\theta=-\frac{1}{\alpha}+\frac{1}{p}+\frac{1}{s}$ and $\q= \nu$ then the Morrey space  $\Mo$ becomes to the space $(L^p(\q),L^s)^\alpha$ in~\cite{Feuto2014}. 

\item In 2014, Liu and Wang \cite{Y.Liu} defined a weighted Morrey space $L^{p,\lambda}_{\alpha, \V,\q}$ associated to Schr\"odinger operator by
\begin{equation*}
\left\| f \right\|_{L_{\alpha ,\V}^{p,\lambda }}^p = \mathop {\sup }\limits_{B\left( {{x_0},r} \right) \subset \R^n } {\left( {1 + \frac{r}{{\rho \left( {{x_0}} \right)}}} \right)^\alpha }{\q(B(x_0,2r))^{ - \lambda }}\int\limits_{B\left( {{x_0},r} \right)} {{{\left| {f\left( x \right)} \right|}^p}\q(x)dx < \infty }.
\end{equation*}
When $s=\infty$ and $\nu = 1$, the space $\Mo$ is also $L_{\alpha p,\V,\q}^{p,n\theta p}$.
\end{itemize}

The first goal of this paper is to prove the boundedness of the Riesz transform $\Ri$ in the Morrey space $\mathbb{M}^{p,s}_{\alpha,\theta}(\q)$, where $\q$ belongs to a class of Muckenhoupt weights $A_{p}$. We then apply this boundedness property to obtain the regularity result of solutions to Schr\"odinger equations $(-\Del +\V)u=f$. 
\begin{theorem}\label{theo1} 
Let $1< p<s\leq \vc,\ \alpha	\in \R, \ \theta \in \left[0,1/p\right)$ and $\q \in A_{p}$. Then for $\V \in RH^\star_{n/2}$, the Riesz transform $\T$ is bounded on $\Mor$, i.e.
\begin{equation}\label{Ri_bound}
	\left\Vert \T f\right\Vert _{\Mor}\ls \left\Vert f\right\Vert _{\Mor},
	\end{equation}
	for all $f \in \Mor$.
\end{theorem}

\begin{theorem}\label{thm1}
Let $1<p<s\leq \vc ,\ \alpha \in \R,\   \theta \in \left[ 0,1/p\right), \ \q \in A_{p}$.
Assume that $\V\in RH^\star_{n/2}$ and $f\in \mathbb{M}^{p,s}_{\alpha,\theta}(\q)$. Let $u$ be a solution to the following equation 
\begin{equation}\label{equ1}
(-\Del +\V)u=f,
\end{equation}
then there exists a positive constant $C$ such that
	\begin{equation*}
	\Vert D^2 u \Vert _{\mathbb{M}^{p,s}_{\alpha,\theta}(\q)}\leq C\left\Vert f\right\Vert _{\mathbb{M}^{p,s}_{\alpha,\theta}(\q)}.
	\end{equation*}
\end{theorem}

In the second goal of this paper, we prove the boundedness property of the $\La$-fractional Riesz transform $\p$ in the new weighted Morrey space $\mathbb{M}^{p,s}_{\alpha,\theta}(\q^p,\q^q)$, where the weight $\q$ belongs to $A_{p,q}$. Finally, we establish the regularity of solutions to Schr\"odinger type equations $(-\Delta +\V)^{\beta}u=f$ in this space.

\begin{theorem}\label{theo3} 
Let $\alpha \in \R$, $\beta \in \left( \frac{1}{2},\frac{n+1}{2}\right)$, $1<p<\frac{n}{2\beta-1}$, $1<s\leq \vc$ and
	$$	\frac{1}{q}=\frac{1}{p}-\frac{2\beta-1}{n},\, \theta \in \left[ 0,1/q\right), \, \q \in A_{p,q}.$$ 
	For any $\V \in RH_{n/2}$, the $\La$-fractional Riesz transform $\p$ is bounded from $ \mathbb{M}^{p,s}_{\alpha,\theta}(\q^p,\q^q)$ into $\mathbb{M}^{q,s}_{\alpha,\theta}(\q^q)$, i.e., there exists a positive constant $C$ such that
	\begin{align}\label{eq:main2}
	\|\p f\|_{\mathbb{M}^{q,s}_{\alpha,\theta}(\q^q)} \le C \|f\|_{\mathbb{M}^{p,s}_{\alpha,\theta}(\q^p,\q^q)},
	\end{align}
	for all $f \in \mathbb{M}^{p,s}_{\alpha,\theta}(\q^p,\q^q)$.
\end{theorem}

\begin{theorem}\label{thm2}
Let $\alpha \in \R , \ \beta \in \left( \frac{1}{2},1\right),\ 1<p<\frac{n}{2\beta-1},\ 1<s\leq \vc$ and
$$	\frac{1}{q}=\frac{1}{p}-\frac{2\beta-1}{n},\ \theta \in \left[ 0,1/q\right),\ \q \in A_{ p,q }.$$ 
Assume that $\V\in RH_{n/2}$ and $f\in \mathbb{M}^{p,s}_{\alpha,\theta}(\q^p,\q^q)$. Let $u$ be a solution to the following equation 
\begin{equation}\label{equ2}
(-\Delta +\V)^{\beta}u=f,
\end{equation}
then there exists a positive constant $C$ such that
	\begin{equation*}
	\left\Vert D u\right\Vert _{\mathbb{M}^{q,s}_{\alpha,\theta}(\q^q)}\leq
	C\left\Vert f\right\Vert _{\mathbb{M}^{p,s}_{\alpha,\theta}(\q^p,\q^q)}.
	\end{equation*}
\end{theorem}

The rest of the paper is organized as follows. In the next section, we present some standard notations and definitions of Muckenhoupt weights and reverse H\"{o}lder classes. We then recall some basic and useful properties of these classes for the convenience of the reader. Moreover, a new generalized weighted Morrey space is also introduced in this section. In Section~\ref{sec:R-trans}, we first prove the boundedness of the $\La$-Riesz transforms $\mathscr{R}_{\La}$ in the weighted Morrey space $\Mor$. Then we apply this boundedness property to get the regularity result for Schr\"{o}dinger type equation~\eqref{equ1}. In the last section, we provide the boundedness of the $\La$-fractional Riesz transforms $\mathscr{R}^{\La}_\beta$ in generalized weighted Morrey space $\Mo$. Finally, we obtain the regularity result for Schr\"{o}dinger type equation~\eqref{equ2} by using the boundedness of the $\La$-fractional Riesz transform $\mathscr{R}^{\La}_\beta$.

\section{Preliminaries}\label{sec:pre}
\subsection{Notations}
We first introduce some nations that we use throughout the paper. For $1\leq p\leq \infty $, we denote by $p^{\prime }$ the H\"older conjugate exponent of $p$, i.e., $ \frac{1}{p} + \frac{1}{p^\prime} = 1.$

Notation $B(x,r)$ denotes a open ball in $\mathbb{R}^n$ with radius $r>0$ and centered at $x \in \mathbb{R}^n$. For each ball $B = B\left(x,r\right)$ in  $\mathbb{R}^{n}$ and for any $\lambda>0$, we set $\lambda B:=B\left( x,\lambda r\right)$, $S_{0}\left( B\right)=B$ and $S_{j}\left( B\right) =2^{j}B\backslash 2^{j-1}B$ for any $j\in \mathbb{N}^*$. 

We denote by $E^c$ and $\chi _{E}$ the complement of the set $E$ in $\mathbb{R}^n$ and its characteristic function, respectively. The average integral of a function $f$ in a measurable subset $E$ of $\mathbb{R}^n$ is defined by
\begin{align*}
\fint_E f(x) dx = \frac{1}{|E|} \int_E f(x) dx,
\end{align*} 
where $|E|$ denotes the Lebesgue measure of $E$.

For a weight $\q,$ we mean that $\q$ is a non-negative measurable and locally integrable function on $\mathbb{R}^n$. For any measurable set $E\subset \mathbb{R}^{n}$ and the weight $\q$, we denote 
$$\q(E):=\int_{E}\q(x)dx.$$

\subsection{Muckenhoupt weights and reverse H\"{o}lder classes}
In this subsection, we first recall the definitions of Muckenhoupt weights $A_{p}$ and the reverse H\"{o}lder classes $RH_{q}$. Then we present some known properties of them which are useful for our results.
\begin{definition}\label{def:Muck}
 For $1<p<\infty $, we say that $\q\in A_{p}$ if there exists a positive constant $C$ such that
\begin{equation*}
\Big( \fint\limits_{B}\q(x)dx\Big) \Big( \fint\limits_{B}\q^{-1/(p-1)}(x)dx\Big)^{p-1}\leq C,
\end{equation*}
for all ball $B$ in $\mathbb{R}^n$. 

For the case $p=1$, we say that $\q \in A_{1}$ if there exists a positive constant $C$
such that for all balls $B\subset \mathbb{R}^n$, 
\begin{equation*}
\fint\limits_{B}\q(y)dy\leq C\q(x),
\end{equation*}%
for a.e. $x\in B$. 

Moreover, we set $A_{\infty }=\bigcup_{p\in \left[ 1,\infty \right) }A_{p}$.
\end{definition}

\begin{definition}\label{def:RH}
For some $1<q<\infty$, we say that the weight $\q$ belongs to reverse H\"{o}lder class $RH_{q}$ if there exists a positive constant $C$ such that for all balls $B\subset X$, 
\begin{equation*}
\Big( \fint\limits_{B}\q^{q}(x)dx\Big)^{1/q}\leq C\fint\limits_{B}\q(x)dx.
\end{equation*}

When $q=\infty $, we say that $\q\in RH_{\infty }$ if there exists a constant $C>0 $ such that for all balls $B\subset X$, 
\begin{equation*}
\q(x)\leq C\fint\limits_{B}\q(y)dy,
\end{equation*}%
for almost everywhere $x\in B$.
\end{definition}

For $\sigma \geq 1$, we say that the weight $\q$ belongs to $D_{\sigma }$ if there exists a constant $C>0$ such that 
\begin{equation}\label{eq:doubling}
\q(tB)\leq Ct^{n\sigma }\q(B)\text{ for all }t\geq 1.
\end{equation}
We note that $\q\in A_{p}$ implies that $\q\in D_{p}.$


Let us introduce a positive function $\rho: \ \mathbb{R}^n \rightarrow \mathbb{R}$ defined by:
\begin{equation}\label{def:rho}
\rho \left( x\right) = \sup \Bigg\{ r>0: \ \frac{1}{r^{n-2}}\int\limits_{B\left( x,r\right) }\V(y) dy\leq
1\Bigg\}, \quad x \in \mathbb{R}^n.
\end{equation}

\begin{definition} We define by $RH^{\star}_{n/2}$ the set of all functions $\V\in RH_{n/2},$ such that there exists a positive constant $C$ not depending to $\V$ such that
\begin{align*}
|D \V(x)| \leq C \rho^{-3}(x) \ \mbox{ and } \ |D^2 \V(x)| \leq C \rho^{-4}(x),
\end{align*}
for all $x\in \R^n$. Here $Du$ and $D^2u$ denote the gradient and the Hessian matrix of a function $u$ respectively. 
\end{definition}

\begin{remark} We remark that
\begin{itemize}
\item[i)] if $|D \V(x)|\leq C \rho^{-3}(x)$ then $\V(x)\leq C \rho^{-2}(x)$ (see \cite{ks2000,Shen1996}); 
\item[ii)] and if
$\V(x)=|P(x)|^\alpha$, where  $\alpha>0$ and $P(x)$ is a polynomial, then $\V \in RH^{\star}_{n/2}$.
\end{itemize}	
\end{remark}

We now recall an important property of the auxiliary function $\rho\left( x\right)$ in the following lemma.

\begin{lemma}[see \cite{Shen1995}]
	\label{lem1} Let $\V \in RH_{q}$ with $q\geq \frac{n}{2}.$ Then there
	exists a positive constant $C$  such that	
		\begin{align}\label{lem:Shen_eq1}
		\rho\left( x\right) \sim \rho \left( y\right) \ \mbox{ if } \ \left\vert
		x-y\right\vert \leq C \rho(x).
		\end{align}
		Moreover, there exists $N_0 \in \mathbb{N}$ such that		
		\begin{align}\label{lem:Shen_eq2}
		C^{-1}\rho(x) \Big(1+\frac{|x-y|}{\rho(x)}\Big)^{-N_0} \leq \rho(y) \leq C\rho(x) \Big(1+\frac{|x-y|}{\rho(x)}\Big)^{\frac{N_0}{N_0+1}}
		\end{align}
		for all $x,y \in \mathbb{R}^n$.
\end{lemma}

\begin{lemma}\label{ntd3}
	Let $\alpha \in \R$, the ball $B=B(y,r)$ in $\mathbb{R}^n$ and the function $\rho$ is given by~\eqref{def:rho}. There holds
	\begin{align}\label{est:rho1}
	(1+r\rho^{-1}(y))^\alpha \ls (1+2r\rho^{-1}(y))^\alpha.
	\end{align}
	 Moreover, for all $m\in \N$, $x\in B$ and $z \in S_j(B)$, there exists $N_0 \in \mathbb{N}$ such that
	\begin{align}\label{est:rho2}
	\left( 1+\left\vert x-z\right\vert\rho^{-1}(x) \right) ^{-m} \ls \left( 1+2^{j}r\rh \right) ^{-\frac{m}{N_0+1}},
	\end{align}
	and
	\begin{align}\label{est:rho3}
	\dfrac{(1+r\rho^{-1}(y))^\alpha}{(1+2^jr\rho^{-1}(y))^{\frac{m}{N_0+1}}}\ls (1+2^jr\rho^{-1}(y))^\alpha, 
	\end{align}
	for every $j \in \N$ and $m>(N_0+1)(|\alpha|+\alpha)$.
\end{lemma}
\begin{proof} 
	It is easy to see that if $\alpha <0$ then 
	\begin{align*}
		\left( 1+r\rh \right) ^{\alpha } = \frac{2^{-\alpha }}{\left(2+2r\rh \right) ^{-\alpha }} \leq \frac{2^{-\alpha }}{\left( 1+2r\rh \right) ^{-\alpha }},
	\end{align*}
	and  $\left( 1+r\rh \right) ^{\alpha }\leq
	\left( 1+2r\rh \right) ^{\alpha },$
	for $\alpha \geq 0$. We thus get for all $\alpha \in \R$,
	\begin{equation*}
	\left( 1+r\rh \right) ^{\alpha }\ls \left(	1+2r\rh \right) ^{\alpha},  
	\end{equation*}
	which leads to~\eqref{est:rho1}. To estimate~\eqref{est:rho2}, we first note that by~\eqref{lem:Shen_eq1} in Lemma~\ref{lem1} for any $x\in B,z\in S_{j}\left( B\right),$ we have $\left\vert x-z\right\vert \sim 2^{j}r$ and $\rho(x)\sim \rho(y)$. Applying~\eqref{lem:Shen_eq2} in Lemma~\ref{lem1}, we then get that	
	\begin{equation*}
	\begin{aligned}
	\left( 1+|x-z|\rho^{-1}(x) \right) ^{-m} &\ls \Big[1+2^{j}r\rh \left( 1+r\rh \right)^{-\frac{N_0}{N_0+1}}\Big] ^{-m}   \\
	& \ls \left[ \frac{\left(1+r\rh\right)^{\frac{N_0}{N_0+1}}}{\left(1+r\rh\right)^{\frac{N_0}{N_0+1}}+2^jr\rh} \right]^m
	\\
	& \ls \left[ \frac{\left(1+2^j r\rh\right)^{\frac{N_0}{N_0+1}}}{1+2^jr\rh} \right]^m
	\\
	&\ls \left( 1+2^{j}r\rh \right) ^{-\frac{m}{N_0+1}},  
	\end{aligned}
	\end{equation*}
	which leads to estimate~\eqref{est:rho2}. Finally, to obtain the inequality~\eqref{est:rho3}, we first remark that
	\begin{equation*}
	\left( 1+r\rh \right) ^{\alpha }\leq \left(	1+2^{j}r\rh \right) ^{\left\vert \alpha \right\vert }
	\end{equation*}
	for all $\alpha \in \R$. It follows that 
	\begin{align}\label{est:rho4}
		\frac{\left( 1+r\rh \right) ^{\alpha }}{\left(1+2^{j}r\rh \right) ^{\frac{m}{N_0+1}}} 
		\leq \frac{\left( 1+2^{j}r\rh \right) ^{\left\vert \alpha \right\vert }}{\left( 1+2^{j}r\rh \right) ^{\frac{m}{N_0+1}}} 
		\leq \frac{\left( 1+2^{j}r\rh \right) ^{\alpha }}{\left( 1+2^{j}r\rh \right) ^{\frac{m}{N_0+1}-\left\vert \alpha \right\vert -\alpha }}.
	\end{align}	
	 By choosing $N_0 \in \mathbb{N}$ such that 
	$$\frac{m}{N_0+1}-\left\vert \alpha \right\vert -\alpha >0,$$
	we deduce estimate~\eqref{est:rho3} from \eqref{est:rho4}. The proof is complete. 	
\end{proof}

Now, we recall some basic properties of the Muckenhoupt weights and the reverse H\"older classes.

\begin{lemma}(See \cite[Lemma 2.1]{Song2010} and \cite[Proposition 7.2.8]{Grafakos1})
	\label{tr1}The following properties hold:	
	\begin{itemize}
		\item[i)] $A_{1}\subset A_{p}\subset A_{q}$ for $1\leq p\leq q\leq \infty $.		
		\item[ii)] $RH_{\infty }\subset RH_{q}\subset RH_{p}$ for $1<p\leq q\leq \infty $.		
		\item[iii)] If $\q \in A_{p},1<p<\infty $, then there exists $\epsilon >0$ such that 
		$\q \in A_{p-\epsilon }$.		
		\item[iv)] If $\q \in RH_{q},1<q<\infty $, then there exists $\epsilon >0$ such
		that $\q \in RH_{q+\epsilon }$.		
		\item[v)] $A_{\infty }=\displaystyle{\bigcup _{1\leq p<\infty }A_{p}}= \displaystyle{\bigcup _{1<p\leq \infty}RH_{p}}.$		
		\item[vi)] There exists $\delta \in \left( 0,1\right) $ such that for any ball $B\subset\mathbb{R}^{n}$ and any measurable subset $E$ of all $B$,
		\begin{equation}\label{nv}
		\frac{\q\left( E\right) }{\q\left( B\right) }\lesssim \Big( \frac{\left\vert
			E\right\vert }{\left\vert B\right\vert }\Big) ^{\delta }  .
		\end{equation}
	\end{itemize}
\end{lemma}

To establish the weighted inequality for fractional integrals, we need to introduce class $A_{p,q}.$ 
\begin{definition}\label{def:Apq}
We say that a weight $\q$ belongs to the class $A_{p,q}$ for  $1\leq p<\infty $ and $1\leq q<\infty $, if there exists a positive constant $C$ such that
\begin{equation*}
\Big( \fint \q^{q}(x)dx\Big) ^{1/q}\Big( \fint \q^{-p^{\prime }}(x)dx\Big)^{1/p^{\prime }}\leq C,
\end{equation*}
 for any ball $B$ in $\mathbb{R}^n$
\end{definition}

We remark that if $\q \in A_{p,q}$ then 
\begin{equation*}
\Big(\fint \q^{q/p}(x)dx\Big) ^{1/q}\Big(\fint \q^{-p^{\prime }/p}(x)dx\Big)^{1/p^{\prime }}\leq C
\end{equation*}%
for any ball $B$ in $\mathbb{R}^n$. The connection of $A_{p,q}$ and $A_{m}$ is also showed in the following lemma.

\begin{lemma}(See \cite[Remark 2.11]{KS2009})
	Let $0<\beta <n$, $1\leq p<\frac{n}{\alpha}$, and $\frac{1}{q} = \frac{1}{p} - \frac{\beta}{n}$. The
	following statements are true:
	
	\begin{itemize}
		\item[i)] For any $p>1$, if $\q \in A_{p,q}$ then $\q^{q}\in A_{q},$ and $%
		\q^{-p^{\prime }}\in A_{p^{\prime }}.$
		
		\item[ii)] $\q \in A_{1,q}$ if and only if $\q^{q}\in A_{1}.$
	\end{itemize}
\end{lemma}

\subsection{A generalized Morrey type space}
Let us now introduce a new Morrey type space which is a generalized version of many well-known Morrey type spaces. 

\begin{definition}\label{morrey1}
	Let $\alpha \in \R , \ \theta \in \left[ 0,1\right),\ 1\leq p<s\leq \vc$ and $\q,\nu \in \Lo(\mathbb{R}^n)$.  We denote by $\Mo$ the space of all measurable functions $f \in L^p_{\operatorname{loc}}(\mathbb{R}^n)$ such that
	\begin{align*}
		\left\Vert f\right\Vert _{\Mo} =\underset{r>0}{\sup }\Big[ \int\limits_{
			\mathbb{R}^{n}}\Big( \left( 1+r \rho^{-1}(x) \right) ^{\alpha }\nu\left(
		B(x,r)\right) ^{-\theta }\left\Vert f \chi _{B(x,r)}\right\Vert _{L^p( \q )
		} \Big) ^{s}dx\Big] ^{1/s}<\vc,
	\end{align*}
	where the function $\rho$ is defined by~\eqref{def:rho}. In the case of $\q \equiv \nu$, we denote $\Mo$ by $\mathbb{M}^{p,s}_{\alpha,\theta}(\q)$ for the simplicity.
\end{definition}

\section{Boundedness of $\La$-Riesz transform}\label{sec:R-trans}

In \cite{Shen1995}, Shen proved that the operator $\Ri$ is a Calder\'on-Zygmund operator if $\V$ is a non negative polynomial and $\Ri$ is bounded in $L^p$ if $\V \in RH_{n/2}$. In this paper, we obtain the general result in the new Morrey space $\Mo$ under the assumption $\V \in RH^\star_{n/2}$. Let us introduce a kernel $\K(x,y)$ associated to operator $\T$ as follows
\begin{align}\label{def:Ri}
\T f(x)=\tp_{\R^n}\K(x,y)f(y)dy.
\end{align}
We next state several lemmas which are useful to prove our main result about the boundedness property of the $\La$-Riesz transform $\Ri$ in Morrey space $\Mo$. The proof of Lemma~\ref{lemma5.1} can be found in \cite[Lemma 3.6]{pt2016}.
\begin{lemma}\label{lemma5.1} Let $\V \in RH^{\star}_{n/2}$. For any $k>0$ there exists a positive constant $C$ such that 	
$$|\K(x,y)| \leq \frac{C}{\Big(1+\frac{|x-y|}{\rho(x)}\Big)^k}\frac{1}{|x-y|^n},$$
for all $x,y\in \R^n,x\neq y.$
\end{lemma}
\begin{lemma}
	\label{ntd} Let $\q \in A_p$, $1 \leq p < \vc$ and $f \in L^p( \q )$. Then
	\begin{equation*}
	\fint\limits_{B}\left\vert f(z)\right\vert dz\leq C \q(B) ^{-1/p}\left\Vert f\chi _{B}\right\Vert _{L^p(\q)},
	\end{equation*}
holds for every ball $B\subset \R^n.$
\end{lemma}
\begin{proof}
	We consider two cases $p=1$ and $p>1$. For the first case $p=1$, since $\q \in A_{1}$ we obtain that
	\begin{equation*}
	\fint\limits_{B}\left\vert f(z)\right\vert dz=\frac{1}{\q(B) }\fint\limits_{B}{\q(B) }f(z)dz\ls \frac{1}{\q(B) }\left\Vert f\chi _{B}\right\Vert _{L^1 (\q) },
	\end{equation*}
	for every ball $B\subset \mathbb{R}^{n}$. For the second case $p>1$, by H\"older's inequality and the definition of $A_{p}$, one has
	\begin{align*}
		\fint\limits_{B}\left\vert f(z)\right\vert dz
		&\leq  \Big(\fint\limits_{B}\left\vert f(z)\right\vert^{p}\q(z)dz\Big)^{1/p}\Big(\fint\limits_{B}\q^{-\frac{1}{p-1}}(z)dz\Big)^{\frac{p-1}{p}} \\
		& \ls  \q(B)^{-1/p}\left\Vert f\chi_{B}\right\Vert_{L^p( \q ) },
	\end{align*}
	for every ball $B\subset \R^n.$ The proof is complete.
\end{proof}

\begin{lemma}\label{ntd1}
Let $1\leq p<s\leq \vc, \ \alpha \in \R, \ \theta \in \left[0,1/p\right)$ and $\q \in A_{p}.$ For any ball $B=B(y,r)$ and $m>0$ let us set 
	\begin{equation}
	\label{eq:F1}
	\mathbb{F}(f,B)=\sum\limits_{j=0}^{\vc }j\left( 1+2^{j}r\rh \right)^{-\frac{m}{N_0+1}}\q\left( 2^{j}B\right) ^{-1/p} \left\Vert f \chi _{2^{j}B}\right \Vert _{L^p( \q ) }.
	\end{equation}
	Then there exists a constant $C>0$ not depending on $B$ such that for all $f \in \Mor$ and $m>(N_0+1)(|\alpha|+\alpha)$, there holds
	\begin{equation}\label{est:rho6}
	\underset{r>0}{\sup }\Big[ \int\limits_{\mathbb{R}^{n}} \Big( \q\left( B\right) ^{1/p-\theta }\left( 1+r\rh \right)^{\alpha }\mathbb{F}(f,B) \Big) ^{s}dy\Big] ^{1/s} \leq C \left\Vert f\right\Vert _{\Mor}.
	\end{equation}
\end{lemma}
\begin{proof}	
	Be the definition of the function $\mathbb{F}$ in~\eqref{eq:F1}, we can estimate
	\begin{align}\nonumber
	A\  :& = \ \q\left( B\right) ^{1/p-\theta }\left( 1+r\rh \right)^{\alpha }\mathbb{F}(f,B)\\ \nonumber
	  & \ls \sum\limits_{j=0}^{\vc }j\left( 1+r\rh \right)^{\alpha }\q\left( B\right) ^{1/p-\theta } \left( 1+2^{j}r\rh \right) ^{-\frac{m}{N_0+1}} \\ \label{est:5} 
	& \qquad \qquad \times \q\left( 2^{j}B\right) ^{-\frac{1}{p}}\left\Vert f\chi _{2^{j}B}\right\Vert_{L^p( \q ) }.
	\end{align}	
	Using the inequality~\eqref{est:rho3} in Lemma~\ref{ntd3}, we obtain from~\eqref{est:5} that
	\begin{align*}
		A &\ls \sum\limits_{j=0}^{\vc }j\left( 1+2^{j}r\rh \right) ^{\alpha }\q\left( B\right) ^{1/p-\theta }\q\left(2^{j}B\right) ^{-\frac{1}{p}}\left\Vert f\chi _{2^{j}B}\right\Vert_{L^p( \q ) } \\
		&\ls \sum\limits_{j=0}^{\vc }j\left( 1+2^{j}r\rh \right) ^{\alpha }\q\left( 2^{j}B\right) ^{-\theta }\Big( \frac{\q\left( B\right) }{\q\left( 2^{j}B\right) }\Big) ^{1/p-\theta }\left\Vert f\chi _{2^{j}B}\right\Vert _{L^p( \q ) }.
	\end{align*}
	Thanks to \eqref{nv} in Lemma~\ref{tr1} and remark  that $1/p-\theta \geq 0$ in the above inequality, we can estimate $A$ as follows
	\begin{equation*}
	A\ls \sum\limits_{j=0}^{\vc }j2^{-j\delta \left( 1/p-\theta \right) }\left( 1+2^{j}r\rh\right) ^{\alpha }\q\left( 2^{j}B\right) ^{-\theta }\left\Vert f\chi_{2^{j}B}\right\Vert _{L^p( \q ) }.
	\end{equation*}	 
	 By the definition of $\left\Vert f\right\Vert_{\Mor}$ in Definition~\ref{morrey1}, we can conclude that
	\begin{equation*}
	\underset{r>0}{\sup }\Big[ \int\limits_{\mathbb{R}^{n}} \Big( \q\left( B\right) ^{1/p-\theta }\left( 1+r\rh \right)^{\alpha }\mathbb{F}(f,B) \Big) ^{s}dy\Big] ^{1/s} \ls\sum\limits_{j=0}^{\vc	}j2^{-j\delta \left( 1/p-\theta \right) } \left\Vert f\right\Vert_{\Mor},
	\end{equation*}
which leads to~\eqref{est:rho6} with noting that	$1/p-\theta \geq 0$. The proof is complete.
\end{proof}

\medskip

\begin{proof}[Proof of Theorem~\ref{theo1}]
	Let $f\in  \Mor$ and the ball $B=B\left( y,r\right) .$ We decompose $f$ by
	\begin{equation*}
	f=f\chi _{2B}+f\chi _{\left( 2B\right) ^{c}}=:f_{1}+f_{2}.
	\end{equation*}
	Thanks to Lemma~\ref{lemma5.1} and the inequality~\eqref{est:rho2} in Lemma~\ref{ntd3}, for all $m\in \mathbb{N}$, one has
	\begin{align}\nonumber
	\left\vert \T(f_2)\left( x\right) \right\vert &\ls \int\limits_{\left( 2B\right) ^{c}}\left\vert \K\left( x,z\right)	\right\vert \left\vert f\left( z\right) \right\vert dz   \\ \nonumber
	&\ls \sum\limits_{j=2}^{\vc }\int\limits_{S_{j}\left( B\right) }\frac{1}{\left( 1+\left\vert x-z\right\vert \rho^{-1}(x) \right) ^{m}}\frac{1}{\left\vert x-z\right\vert ^{n}}\left\vert f\left( z\right)
	\right\vert dz   \\ \label{l1}
	&\ls \sum\limits_{j=2}^{\vc }\left( 1+2^{j}r\rh \right)^{-\frac{m}{N_0+1}}\frac{1}{\left( 2^{j}r\right) ^{n}}\int \limits_{2^{j}B}\left\vert f(z)\right\vert dz.  
	\end{align}	
	Applying Lemma~\ref{ntd}, we obtain from~\eqref{l1} that	
	\begin{align*}
	\left\vert \T(f_2)\left( x\right) \right\vert\ls\sum\limits_{j=2}^{\vc }\left( 1+2^{j}r\rh \right)^{-\frac{m}{N_0+1}} \times \q\left( 2^{j}B\right) ^{-1/p}   \left\Vert f \chi _{2^{j}B}\right \Vert _{L^p( \q ) },
	\end{align*}	
	which leads to
	\begin{align} \nonumber
	\left\vert \T f\left( x\right) \right\vert &\ls \left\vert \T\left(f_{1}\right) \left( x\right) \right\vert +\sum\limits_{j=0}^{\vc }\left(1+2^{j}r\rh \right) ^{-\frac{m}{N_0+1}}\times \q \left( 2^{j}B\right) ^{-1/p}\left\Vert f\chi _{2^{j}B}\right\Vert_{L^p( \q ) }\\ \label{b3}
	& \ls |\T f_1 (x)|+ \mathbb{F} (f,B),
	\end{align}
for almost everywhere  $x\in B$. From \eqref{b3} and the boundedness of the Riesz transform $\T$ in the weighted Lebesgue space $L^p(\q)$, we get that
	\begin{align}\label{eq:Tf}
\left\Vert \T(f) \chi _{B}\right\Vert _{L^p( \q )}\ls \left\Vert f\chi_{B}\right\Vert _{L^p( \q )}+\q(B)^{1/p}\mathbb{F}(f,B).
	\end{align}	
Multiplying two sides of \eqref{eq:Tf} by $\q\left( B\right)^{-\theta }\left( 1+r\rh \right) ^{\alpha },$ one has
	\begin{align*}
	\q\left( B\right) ^{-\theta }\left( 1+r\rh \right)^{\alpha } & \left\Vert \T(f) \chi_{B}\right\Vert _{L^{p}\left(\q\right) } \ls J_1 + J_2,
	\end{align*}
where	
\begin{align*}
J_1 =  \left\Vert f\chi _{2B}\right\Vert _{L^p( \q ) }\q\left(B\right) ^{-\theta }\left( 1+r\rh \right) ^{\alpha},
\end{align*}
and
\begin{align*}
J_2 = \q\left( B\right) ^{1/p-\theta }\left( 1+r\rh \right)^{\alpha }\mathbb{F}(f,B).
\end{align*}	
Applying \eqref{est:rho1} in Lemma \ref{ntd3}, we can estimate $J_1$ as
\begin{align*}
J_{1} &\ls \q\left( 2B\right) ^{-\theta }\left( 1+2r\rh \right) ^{\alpha }\left\Vert f\chi _{2B}\right\Vert_{L^p( \q ) }\Big( \frac{\q\left( 2B\right) }{\q\left( B\right) }\Big) ^{\theta } \\
&\ls \q\left( 2B\right) ^{-\theta }\left( 1+2r\rh\right) ^{\alpha }\left\Vert f\chi _{2B}\right\Vert _{L^p( \q ) },
\end{align*}
where we use the doubling property~\eqref{eq:doubling} of $\q$ in the last inequality. Combining this estimate and inequality~\eqref{est:rho6} in Lemma \ref{ntd1}, there exists a positive constant $C$ such that
	\begin{equation*}
	\underset{r>0}{\sup }\Big[ \int\limits_{
		\mathbb{R}^{n}} \Big( \q\left( B\right) ^{-\theta }\left( 1+r\rh \right)
	^{\alpha }\left\Vert \T(f) \chi_{B}\right\Vert _{L^{p}\left(
	\q\right) } \Big) ^{s}dy\Big] ^{1/s}  \leq C \left\Vert f\right\Vert _{\Mor}.
	\end{equation*}
Finally, by the definition of $\left\Vert \T f\right\Vert _{\Mor}$, we can conclude that
	\begin{equation*}
	\left\Vert \T f\right\Vert _{\Mor}\ls \left\Vert f\right\Vert _{\Mor},
	\end{equation*}
	which finishes the proof.
\end{proof}

\medskip

\begin{proof}[Proof of Theorem \ref{thm1}]
 The boundedness property of $D^2 u$ can be obtained by the boundedness of $\T f$ in Theorem~\ref{theo1}.
\end{proof}

\section{Boundedness of $\La$-fractional Riesz transform}\label{sec:R-potential}

To consider $\La$-fractional Riesz transform $\mathscr{R}^{\La}_\beta=D \La^{-\beta}$, we recall the classical Riesz potential. In 1974, Muckenhoupt and Wheeden \cite{mu} proposed the boundedness property for the classical Riesz potential $\mathbb{I}_{\beta }$ defined by
\begin{align}\label{def:Ri-potential}
\mathbb{I}_\beta f(x)= \int_{\R^n}\frac{f(y)}{|x-y|^{n-\beta}}dy, 
\end{align}
in weighted Lebesgue space $L^p(\q)$. Their result is stated in the next lemma. 
\begin{lemma}\label{thn}
	\label{mu1}Let $0<\beta <n,\ 1<p<\frac{n}{\beta}$, $\frac{1}{q} = \frac{1}{p} - \frac{\beta}{n}$ and $\q\in
	A_{p,q}$. Then the Riesz potential $\mathbb{I}_{\beta }$ is bounded from $L^{p}\left( \q^{p}\right)$ into $L^{q}\left( \q^{q}\right)$, i.e., there exists a positive constant $C$ such that
	\begin{equation*}
	\left\Vert \mathbb{I}_{\beta } f\right\Vert _{L^{q}\left( \q^{q}\right) }\leq
	C\left\Vert f\right\Vert _{L^{p}\left( \q^{p}\right) },
	\end{equation*}
	for all $f \in L^{p}\left( \q^{p}\right)$.
\end{lemma}

We denote by $\K_\beta(x,y)$ the kernel associated to the $\La$-Riesz potential $\p$. An estimate of the kernel $\K_\beta(x,y)$ is directly obtained by using estimation $(11)$ in \cite[page 241]{Y.Liu1}.

\begin{proposition} \label{td10}
	Let $\frac{1}{2}<\beta\leq 1, \ m\in \N,\ m\geq 2$. There exists $C_{m}>0$ such that for any ball $B$, for all $x\in B, \ y\in (2B)^{c}$ there holds
	\begin{equation}
	|\K_\beta(x,y)| \leq \frac{C_{m}}{\left( 1+\left\vert x-y\right\vert\rho^{-1}(x) \right) ^{m}}\frac{1}{\left\vert x-y\right\vert^{n-\beta_1 }}, 
	\end{equation}
	where $\beta_1=2\beta-1.$
\end{proposition} 

Combining the definition of the Riesz potential $\mathbb{I}_\beta$ in \eqref{def:Ri-potential}, Lemma~\ref{thn} and Proposition~\ref{td10} we may obtain the next lemma.

\begin{lemma}\label{thn1}
Let $\alpha \in \R$,  $\beta \in \left( \frac{1}{2},\frac{n+1}{2}\right), \ 1<p<\frac{n}{2\beta-1}, \ 1<s\leq \vc$ and
$$\frac{1}{q}=\frac{1}{p}-\frac{2\beta-1}{n},\theta \in \left[ 0,1/q\right).$$ 
For any $\q\in A_{p,q}$, the $\La$-fractional Riesz transform $\p$ is bounded from $L^{p}\left( \q^{p}\right)$ into $L^{q}\left( \q^{q}\right)$, i.e., there exists a positive constant $C$ such that
	\begin{equation*}
	\left\Vert \p f\right\Vert _{L^{q}\left( \q^{q}\right) }\leq C\left\Vert f\right\Vert _{L^{p}\left( \q^{p}\right) },
	\end{equation*}
	for all $f \in L^{p}\left( \q^{p}\right)$.
\end{lemma}

We now proof the following lemma.
\begin{lemma}\label{ntd2}
Let $\alpha \in \R$,  $\beta \in \left( \frac{1}{2},\frac{n+1}{2}\right), \, 1<p<\frac{n}{2\beta-1}, \, 1<s\leq \vc$ and
$$\frac{1}{q}=\frac{1}{p}-\frac{2\beta-1}{n},\, \theta \in \left[ 0,1/q\right),\, \q \in A_{ p,q }.$$ 
For any ball $B=B(y,r)$ and $m>(N_0+1)(|\alpha|+\alpha)$, there exists $C>0$ not depending on $B$ such that for all $f \in \mathbb{M}^{p,s}_{\alpha,\theta}(\q^p,\q^q)$ there holds
	\begin{equation*}
	\underset{r>0}{\sup }\Big[ \int\limits_{\mathbb{R}^{n}} \Big( \q^{q}\left( B\right) ^{1/q-\theta }\left( 1+r\rh
	\right) ^{\alpha }\mathbb{W}(f,B) \Big) ^{s}dy\Big] ^{1/s}   \leq C \left\Vert f\right\Vert _{\mathbb{M}^{p,s}_{\alpha,\theta}(\q^p,\q^q)},
	\end{equation*}
	where the function $\mathbb{W}(f,B)$ is defined by
	\begin{align}\label{WfB}
	\mathbb{W}(f,B)=\sum\limits_{j=0}^{\vc }j\left( 1+2^{j}r\rh \right)^{-\frac{m}{N_0+1}}\q^q \left( 2^{j}B\right) ^{-1/q}   \left\Vert f \chi _{2^{j}B}\right \Vert _{L^p( \q^p ) }.
	\end{align}
\end{lemma}
\begin{proof}
	By~\eqref{est:rho3} in Lemma~\ref{ntd3}, we can estimate
	\begin{align}\nonumber
	A&:=\q^{q}\left( B\right) ^{1/q-\theta }\left( 1+r\rh\right) ^{\alpha }\mathbb{W}(f,B)\\ \nonumber
	&\ls \q^{q}\left( B\right) ^{-\theta+1/q}\sum\limits_{j=0}^{\vc }\left( 1+2^{j}r\rh\right) ^{\alpha }\left\Vert f\chi _{2^{j}B}\right\Vert _{L^{p}\left(\q^p\right) }\q^{q}\left( 2^{j}B\right) ^{-1/q} \\ \label{est:A1}
	&\ls \sum\limits_{j=2}^{\vc }\Big( \frac{\q^{q}\left( B\right) }{\q^{q}\left( 2^{j}B\right) }\Big) ^{1/q-\theta }\left(1+2^{j}r\rh \right) ^{\alpha }\q^{q}\left(2^{j}B\right) ^{-\theta }\left\Vert f\chi _{2^{j}B}\right\Vert _{L^{p}\left(\q^p\right) }.
	\end{align}
Thanks to estimate \eqref{nv} and noting that $1/q-\theta \geq 0$, one deduces from \eqref{est:A1} that
	\begin{equation}\label{est:A2}
	A\ls \sum\limits_{j=0}^{\vc }\frac{1}{2^{j\delta n\left( 1/q-\theta\right) }}\left( 1+2^{j}r\rh \right) ^{\alpha
}\q^{q}\left( 2^{j}B\right) ^{-\theta }\left\Vert f\chi _{2^{j}B}\right\Vert_{L^{p}\left( \q^p\right) }.
	\end{equation}	
	Taking the supremum both sides of~\eqref{est:A2}, by the definitions of $\left\Vert \cdot\right\Vert _{\mathbb{M}^{p,s}_{\alpha,\theta}(\q^p,\q^q)}$, we obtain that
	\begin{align*}
	 \underset{r>0}{\sup }\Big[ \int\limits_{\mathbb{R}^{n}} \Big( \q^{q}\left( B\right) ^{1/q-\theta }\left( 1+r\rh\right) ^{\alpha }\mathbb{W}(f,B) \Big) ^{s}dy\Big] ^{1/s}  \ls \sum\limits_{j=0}^{\vc}2 ^{-j\delta n( 1/q-\theta) }
\left\Vert f\right\Vert _{\mathbb{M}^{p,s}_{\alpha,\theta}(\q^p,\q^q)}.
	\end{align*}	
	Since $1/q-\theta>0 $ there holds
	\begin{equation*}
	\underset{r>0}{\sup }\Big[ \int\limits_{
		\mathbb{R}^{n}} \Big( \q^{q}\left( B\right) ^{1/q-\theta }\left( 1+r\rh
	\right) ^{\alpha }\mathbb{W}(f,B) \Big) ^{s}dy\Big] ^{1/s} \ls \left\Vert f\right\Vert _{\mathbb{M}^{p,s}_{\alpha,\theta}(\q^p,\q^q)},
	\end{equation*}
	which completes the proof.
\end{proof}

\medskip

\begin{proof}[Proof of Theorem~\ref{theo3}]
	Let $B=B(y,r)$ be a ball in  $\R^n$ and $f \in \mathbb{M}^{p,s}_{\alpha,\theta}(\q^p,\q^q)$. We decompose $f$ as follows
	\begin{equation*}
	f=f\chi _{2B}+f\chi _{(2B)^{c}}=:f_{1}+f_{2}.
	\end{equation*}
	By the linearity of $\p$, one has
	\begin{align*}
	\q^{q}\left( B\right) ^{-\theta } & \left( 1+r\rh \right)^{\alpha }  \Big( \int\limits_{B}\left\vert \p f(x)\right\vert	^{q}\q^{q}(x)dx\Big) ^{1/q} \\
	&\leq \q^{q}\left( B\right) ^{-\theta }\left( 1+r\rh\right) ^{\alpha }\Big( \int\limits_{B}\left\vert \p f_{1}(x)\right\vert ^{q}\q^{q}(x)dx\Big) ^{1/q} \\
	& \qquad \qquad +\q^{q}\left( B\right) ^{-\theta }\left( 1+r\rh\right) ^{\alpha }\Big( \int\limits_{B}\left\vert \p f_{2}(x)\right\vert ^{q}\q^{q}(x)dx\Big) ^{1/q}\\
	& =:  J_{1}+J_{2}.
	\end{align*}
We can estimate $J_1$ by the boundedness from $L^{p}\left( \q^{q}\right)$ into $L^{p}\left( \q^p\right)$ of $\p$ in Lemma~\ref{mu1}, we have 
	\begin{align} \nonumber
		J_{1} &\leq \q^{q}\left( B\right) ^{-\theta }\left( 1+r\rh \right) ^{\alpha }\Big( \int\limits_{2B}\left\vert f(x)\right\vert ^{p}\q^p(x)dx\Big) ^{1/p} \\ \label{est:A3}
		&\leq \Big( \frac{\q^{q}\left( 2B\right) }{\q^{q}\left( B\right) }\Big)^{\theta }\q^{q}\left( 2B\right) ^{-\theta }\left( 1+2r\rh \right) ^{\alpha }\Big( \int\limits_{2B}\left\vert f(x)\right\vert ^{p}\q^p(x)dx\Big) ^{1/p},
	\end{align}
where the last inequality is obtained by inequality~\eqref{est:rho1} in Lemma~\ref{ntd3}. From~\eqref{est:A3}, the doubling property of $\q^{q}$ gives us
	\begin{align}\label{nnt1}
	J_{1}& \ls \q^{q}\left( 2B\right) ^{-\theta }\left( 1+2r\rh \right) ^{\alpha }\Big( \int\limits_{2B}\left\vert
	f(x)\right\vert ^{p}\q^p(x)dx\Big) ^{1/p}.  
	\end{align}
	Thanks to Proposition~\ref{td10}, for every $m\in \N,\, m\geq 2$, there exists $C_{m}>0$ such that for all $x\in B, \, z\in (2B)^c$ there holds
	\begin{equation}\label{est:Kbeta}
	|\K_\beta (x,z)|\leq \frac{C_{m}}{\left( 1+\left\vert x-z\right\vert
		\rho^{-1}(x) \right) ^{m}}\frac{1}{\left\vert x-z\right\vert
		^{n-\beta_1 }}, 
	\end{equation}
	where $\beta_1=2\beta-1$. For $x\in B$ and $z\in S_{j}\left( B\right)$, we see that $\left\vert x-z\right\vert
	\sim 2^{j}r.$ Combining~\eqref{est:Kbeta} and~\eqref{est:rho2} in Lemma~\ref{ntd3}, one obtains
	\begin{align}\nonumber 
	\left\vert \p\left( f_{2}\right) \left( x\right) \right\vert &\leq \int\limits_{\left( 2B\right) ^{c}}\left\vert f\left( z\right) \right\vert |\K_\beta(x,z)|dz  \\ \label{dt2}
	&\ls \sum\limits_{j=1}^{\vc }\frac{1}{\left(1+2^{j}r\rho^{-1}(x) \right) ^{m}}\frac{1}{\left(2^{j}r\right) ^{n-\beta_1 }}\int\limits_{2^{j}B}\left\vert f(z)\right\vert dz, 
	\end{align}
	for all $m\in \N, \, m \geq 2$. Using H\"older's inequality and assumption $A_{p,q}$ of $\q$ it follows that
	\begin{align}\nonumber
	 \frac{1}{\left( 2^{j}r\right) ^{n-\beta_1 }}\int\limits_{2^{j}B}\left\vert f(z)\right\vert dz 	&\ls \frac{1}{\left( 2^{j}r\right) ^{n-\beta_1 }}\left\Vert f\chi _{2^{j}B}\right\Vert _{L^{p}\left( \q^p\right)}\q^{-p^{\prime }}\left( 2^{j}B\right) ^{1/p^{\prime }}   \\ \nonumber
	&\ls \frac{\left\vert 2^{j}B\right\vert ^{1/q+1/p^{\prime }}}{\left(2^{j}r\right) ^{n-\beta_1 }}\left\Vert f\chi _{2^{j}B}\right\Vert_{L^{p}\left( \q^p\right) }\q^{q}\left( 2^{j}B\right) ^{-1/q}   \\ \label{trt}
	&\ls \left\Vert f\chi _{2^{j}B}\right\Vert _{L^{p}\left( \q^p\right)}\q^{q}\left( 2^{j}B\right) ^{-1/q}.  
	\end{align}	
One implies from~\eqref{l1} and~\eqref{trt} that
	\begin{equation}\label{Shen1995}
	\begin{aligned}
	\left\vert \p\left( f_{2}\right) \left( x\right) \right\vert \ls \sum\limits_{j=0}^{\vc }\left( 1+2^{j}r\rh\right) ^{-\frac{m}{N_0+1}}\left\Vert f\chi _{2^{j}B}\right\Vert_{L^{p}\left( \q^p\right) }\q^{q}\left( 2^{j}B\right) ^{-1/q}  \ls \mathbb{W}(f,B),
	\end{aligned}  
	\end{equation}	
which guarantees the estimate of $J_2$ as follows
	\begin{align*}
	J_2 & \ls  \q^{q}\left( B\right) ^{1/q-\theta }\left( 1+r\rh \right) ^{\alpha }\mathbb{W}(f,B).
	\end{align*}	
	The proof is complete by combining this inequality to Lemma~\ref{ntd2}, estimate~\eqref{nnt1} and the definition of $\|\p f\|_{\mathbb{M}^{q,s}_{\alpha,\theta}(\q^q)}$.
\end{proof}

\medskip

\begin{proof}[Proof of Theorem~\ref{thm2}]
The boundedness property of $D u$ can be obtained by the boundedness of $\p f$ in Theorem~\ref{theo3}.
\end{proof}


\end{document}